\documentclass[12pt,bezier]{article}
\usepackage{mathrsfs}
\usepackage{amsmath}
\usepackage{amsfonts,amsthm,amssymb}
\usepackage{amsfonts}
\usepackage{graphics}
\usepackage{graphicx}
\newtheorem{theorem}{Theorem}[section]
\newtheorem{proposition}[theorem]{Proposition}
\newtheorem{lemma}[theorem]{Lemma}
\newtheorem{corollary}[theorem]{Corollary}

\newtheorem{remark}[theorem]{Remark}

\newtheorem{prelem}{{\bf Proposition}}

\begin{document}

\title{On the total Rainbow domination of digraphs}
\author{{Zhihong Xie}\\[1mm]
\small\emph{College of Science, East China University of Technology,}\\[-2mm]
\small\emph{Nanchang 330013, Jiangxi, P.R.China}\\[-2mm]
\small\tt{E-mail: xiezh168@163.com}\\[-1mm]
}

\date{}

\maketitle {\flushleft Abstract:} {\small For a positive integer $k$, a $k$-rainbow dominating function ($k$RDF) on a digraph $D$ is a function $f$ from
the vertex set $V(D)$ to the set of all subsets of $\{1,2,\ldots,k\}$ such that
for any vertex $v$ with $f(v)=\emptyset$, $\bigcup_{u\in N^-(v)}f(u)=\{1,2,\ldots,k\}$,
where $N^-(v)$ is the set of in-neighbors of $v$.
The weight of a $k$RDF $f$ is defined as $\sum_{v\in V(D)}|f(v)|$.
A $k$RDF $f$ on $D$ with no isolated vertex is called a total $k$-rainbow
dominating function if the subdigraph of $D$ induced by the set $\{v\in V(D):f(v)\ne\emptyset\}$ has
no isolated vertex.
The total $k$-rainbow domination number is the minimum weight of a total $k$-rainbow dominating function on $D$.
In this paper, we establish some bounds for the total $k$-rainbow domination number and we give the total $k$-rainbow domination number of some digraphs.


\section{\large Introduction}

Throughout this paper, $D=(V(D), A(D))$ denotes a finite digraph with neither loops nor
multiple arcs (but pairs of opposite arcs are allowed).
For two vertices $u,v\in V(D)$, we use $(u,v)$ to denote the arc with direction from $u$ to $v$, that is, $u$ is adjacent
to $v$, and we also call $v$ an \emph{out-neighbor} of $u$ and $u$ an \emph{in-neighbor} of $v$.
For $v\in V(D)$, the \emph{out-neighborhood} and \emph{in-neighborhood} of $v$,
denoted by $N^+_D(v)=N^+(v)$ and $N^-_D(v)=N^-(v)$, are the sets of out-neighbors and in-neighbors of $v$, respectively.
The \emph{closed out-neighborhood} and \emph{closed in-neighborhood} of a vertex $v\in V(D)$
are the sets $N_D^+[v]=N^+[v]=N^+(v)\cup \{v\}$ and $N_D^-[v]=N^-[v]=N^-(v)\cup \{v\}$, respectively.
In general, for a set $X\subseteq V(D)$, we denote $N^+_D(X)=N^+(X)=\bigcup_{v\in X}N^+(v)$ and $N^+_D[X]=N^+[X]=\bigcup_{v\in X}N^+[v]$.
The \emph{out-degree} and \emph{in-degree} of a vertex $v\in V(D)$ are defined by
$d^+_D(v)=d^+(v)=|N^+(v)|$ and $d^-_D(v)=d^-(v)=|N^-(v)|$, respectively.
The \emph{maximum out-degree}, \emph{maximum in-degree}, \emph{minimum out-degree} and \emph{minimum in-degree} among the vertices of $D$
are denoted by $\Delta^+(D)=\Delta^+$, $\Delta^-(D)=\Delta^-$, $\delta^+(D)=\delta^+$ and $\delta^-(D)=\delta^-$, respectively.

A \emph{directed star} $S_n$ is a digraph of order $n\ge2$ with vertex set $\{v_0,v_1,\dots,v_{n-1}\}$ and arc set $\{(v_0,v_i):1\le i\le n-1\}$, where $v_0$ is the \emph{center} of $S_n$.
A \emph{leaf} of a digraph $D$ is a vertex of out-degree $0$ and in-degree $1$, while a \emph{support vertex} of $D$ is a vertex adjacent to a leaf.
A set $S\subseteq V(D)$ is a \emph{packing} of a digraph $D$ if $N^+[u]\cap N^+[v]=\emptyset$ for any two
distinct vertices $u,v\in S$.
The \emph{private neighborhood pn}$(v,S)$ of $v \in S\subseteq V(D)$ is defined by \emph{pn}$(v,S)=N^+(v)\backslash N^+(S\backslash\{v\})$. Each
vertex in \emph{pn}$(v,S)$ is called a \emph{private neighbor} of $v$.
For two vertex subsets $X$ and $Y$ of a digraph $D$, if $Y\subseteq N^+[X]$, then we say that $X$ dominates $Y$.
For a subset $S$ of vertices of $D$, we denote by $D[S]$ the subdigraph induced by $S$.

Let $D_1=(V_1, A_1)$ and $D_2=(V_2, A_2)$ be two digraphs with
disjoint vertex sets $V_1$ and $V_2$ and disjoint arc sets $A_1$ and $A_2$, respectively. The
{\em Cartesian product} $D_1\square D_2$ is the digraph with vertex set $V_1\times V_2$ and for
$(x_1, y_1),(x_2, y_2)\in V(D_1\square D_2)$, $((x_1, y_1),(x_2, y_2))\in A(D_1\square D_2)$ if and only if either $(x_1,x_2)\in A_1$ and $y_1=y_2$,
or $x_1=x_2$ and $(y_1,y_2)\in A_2$.

A subset $S$ of vertices of a digraph $D$ is a \emph{dominating set} of $D$ if $N^+[S]=V(D)$.
The \emph{domination number} of a
digraph $D$, denoted by $\gamma(D)$, is the minimum cardinality of a dominating set of $D$.
A dominating set of $D$ of cardinality $\gamma(D)$ is called a $\gamma(D)$-\emph{set}.
The domination number of digraphs was introduced by Fu \cite{Fu}.

A dominating set $S$ of a digraph $D$ is called a \emph{total
dominating set} (TD-set) of $D$ if the subdigraph of
$D$ induced by $S$ has no isolated vertex. The \emph{total
domination number} of a digraph $D$, denoted by $\gamma_t(D)$, is
the minimum cardinality of a TD-set of $D$. A TD-set of $D$ of
cardinality $\gamma_t(D)$ is called a $\gamma_t(D)$-\emph{set}.
The total domination number of digraphs was introduced by Arumugam et al. \cite{Arumugam}, results on which could be found, for examples, in \cite{Hao,Hao1}.

Let $G$ be a (undirected) graph with vertex set $V(G)$ and edge set $E(G)$.
The \emph{open neighborhood} of a vertex $v$ in $G$ is the set $N(v)=\{u\in V(G):uv\in E(G)\}$.
For a positive integer $k$, we use $\mathcal{P}(\{1,2,\ldots,k\})$ to denote the set of all subsets of the set $\{1,2,\ldots,k\}$.
A $k$-\emph{rainbow dominating function} ($k$RDF) on a graph $G$ is a function $f:V(G)\rightarrow \mathcal{P}(\{1,2,\ldots,k\})$
such that for any vertex $v$ with $f(v)=\emptyset$, $\bigcup_{u\in N(v)}f(u)=\{1,2,\ldots,k\}$.
The \emph{weight} of a $k$RDF $f$ is the value $\omega(f)=\sum_{v\in V(G)}|f(v)|$.
The $k$-\emph{rainbow domination number} of a graph $G$, denoted by $\gamma_{rk}(G)$, is the minimum weight of a $k$RDF on $G$.
The concept of rainbow domination in graphs was introduced by Bre\v{s}ar et al. \cite{Bresar} and has been studied by several authors \cite{Meierling,Shao,Sheikholeslami2,Wu,Xu}.

Amjadi et al. \cite{Amjadi} extended the concept of  rainbow domination in graphs to digraphs.
Let $k$ be a positive integer.
A $k$-\emph{rainbow dominating function} ($k$RDF) on a digraph $D$ is a function $f:V(D)\rightarrow \mathcal{P}(\{1,2,\ldots,k\})$
such that for any vertex $v$ with $f(v)=\emptyset$, $\bigcup_{u\in N^-(v)}f(u)=\{1,2,\ldots,k\}$.
The \emph{weight} of a $k$RDF $f$ is the value $\omega(f)=\sum_{v\in V(D)}|f(v)|$.
The $k$-\emph{rainbow domination number} of a digraph $D$, denoted by $\gamma_{rk}(D)$, is the minimum weight of a $k$RDF on $D$.
A $\gamma_{rk}(D)$-\emph{function} is a $k$RDF on $D$ with weight $\gamma_{rk}(D)$.
Rainbow domination in digraphs was studied in \cite{Hao2,Sheikholeslami} and elsewhere.

Recently, Ahangar et al. \cite{Ahangar1} proposed a variant of rainbow domination in graphs, namely, total $k$-rainbow domination.
The \emph{total $k$-rainbow dominating function} (T$k$RDF) on a graph $G$ with no isolated vertex is a $k$RDF $f$ on $G$ with the additional property
that the subgraph of $G$ induced by the set $\{v\in V(G):f(v)\ne\emptyset\}$
has no isolated vertex. The \emph{total $k$-rainbow domination number} $\gamma_{trk}(G)$ is the minimum
weight of a T$k$RDF on $G$.

In this paper, motivated by the work in \cite{Ahangar1}, we initiate the study of total $k$-rainbow domination in digraphs.
A $k$RDF $f$ on a digraph $D$ with no isolated vertex is called a
\emph{total $k$-rainbow dominating function} (T$k$RDF) if the subdigraph of $D$ induced by the set $\{v\in V(D):f(v)\ne\emptyset\}$
has no isolated vertex. The \emph{total $k$-rainbow domination number} $\gamma_{trk}(D)$ is the minimum
weight of a T$k$RDF on $D$. A T$k$RDF on $D$ with weight $\gamma_{trk}(D)$ is called a $\gamma_{trk}(D)$-\emph{function}.
Note that the total $k$-rainbow domination is a generalization of total domination since  $\gamma_{trk}(D)$ is the total domination number $\gamma_t(D)$ when $k=1$.

The rest of the paper is organized as follows. In the next section, we relate the total $k$-rainbow domination number of digraphs to other domination
parameters such as domination number, total domination number and rainbow domination number.
In Sect. 3, we establish sharp bounds on the total $k$-rainbow domination number of a digraph.
Finally, in Sect. 4, we determine the total $k$-rainbow domination number of the Cartesian product of directed paths.

\section{Relations to other domination parameters}

In this section, we shall relate the total $k$-rainbow domination number of digraphs
to other domination parameters such as domination number, total domination number and rainbow domination number.

\begin{theorem}\label{t2}
Let $k$ be a positive integer and let $D$ be a digraph with no isolated vertex.
Then $$\gamma_{trk}(D)\le (k+1)\gamma(D).$$
Moreover, if $\gamma_{trk}(D)=(k+1)\gamma(D)$, then every $\gamma(D)$-set is a packing in $D$.
\end{theorem}

\begin{proof}
Let $S$ be an arbitrary $\gamma(D)$-set and let $S'$ denote the set of vertices
in $S$ that are isolated in $D[S]$ (possibly, $S'=\emptyset$).
For each vertex $v\in S'$, we select
an out-neighbor or in-neighbor of $v$ and denote it by $v'$. Let $S''=\bigcup_{v\in S'}\{v'\}$. Clearly, $S''\subseteq V(D)\backslash S$ and $|S''|\le|S'|\le|S|$.
Now define the function $f:V(D)\rightarrow\mathcal{P}(\{1,2,\ldots,k\})$ by $f(x)=\{1,2,\dots,k\}$ for each $x\in S$, $f(x)=\{1\}$ for each $x\in S''$ and $f(x)=\emptyset$ otherwise.
Then $f$ is a T$k$RDF on $D$ and so
\begin{equation}\label{e2}
\gamma_{trk}(D)\leq k|S|+|S''|\leq k|S|+|S'|\leq (k+1)|S|=(k+1)\gamma(D),
\end{equation}
establishing the desired upper bound.

Suppose that $\gamma_{trk}(D)=(k+1)\gamma(D).$ Then we must have equality throughout the inequality chain (\ref{e2}).
In particular, $|S|=|S'|=|S''|$, implying that $S'=S$ and so $D[S]$ is empty.

We next show that $S$ is a packing in $D$. Suppose, to the contrary, that there exist two
distinct vertices $u$ and $v$ in $S$ such that $N^+[u]\cap N^+[v]\ne\emptyset$.
Let $w\in N^+[u]\cap N^+[v]$.
Since $u,v\in S$ and $D[S]$ is empty, $w\in N^+(u)\cap N^+(v)$.
We now choose $u'=v'=w$ where, as before, $u'$ and $v'$ are the
out-neighbors of $u$ and $v$, respectively.
With this choice of $u'$ and
$v'$, we note that $|S''|<|S'|$, a contradiction.
Hence for every pair of distinct vertices $u$ and $v$ in $S$,
$N^+[u]\cap N^+[v]=\emptyset$, implying that $S$ is a packing in $D$. Thus,
every $\gamma(D)$-set is a packing in $D$.
\end{proof}

\begin{remark}
It should be mentioned that if $D$ is a digraph with no isolated
vertex such that each $\gamma(D)$-set is a packing in $D$, then it is not necessarily true
that $\gamma_{trk}(D)=(k+1)\gamma(D)$.
For example, for $t\ge k+1$, let $D$ be the digraph obtained from a directed star $S_t$ with center $v$ by adding two vertices $x$ and $y$ and two arcs $(x,y)$ and $(y,v)$.
One can verify that $\gamma(D)=2$ and the set $\{x,v\}$ is the unique
$\gamma(D)$-set, which is also a packing in $D$.

However, the function $f$ defined by $f(v)=\{1,2,\dots,k\}$, $f(u)=\{1\}$ for $u\in\{x,y\}$ and $f(u)=\emptyset$ for $u\in V(D)\backslash\{v,x,y\}$, is a T$k$RDF on $D$, implying that
$\gamma_{trk}(D)\le\omega(f)=k+2<2(k+1)=(k+1)\gamma(D)$.
\end{remark}

\begin{remark}
The upper bound of Theorem \ref{t2} is sharp. Let $t\ge2$ and
let $D$ denote the digraph obtained from the disjoint union of $t$ directed stars $S_{i_1}, S_{i_2}, \dots, S_{i_t}$, where $i_j\ge k+2$ for $1\le j\le t$,
by selecting one leaf from every directed star and adding any number of arcs joining these $t$
selected leaves so that the resulting digraph is connected.
It is not difficult to verify that $\gamma(D)=t$.

Let $f$ be an arbitrary $\gamma_{trk}(D)$-function and let $v$ be an arbitrary support vertex  of $D$.
One can check that $\sum_{x\in N^+[v]}|f(x)|\ge k+1$ and so $\gamma_{trk}(D)\ge(k+1)t=(k+1)\gamma(D)$.
Moreover, by Theorem \ref{t2}, $\gamma_{trk}(D)\le(k+1)\gamma(D)$.
Therefore, we obtain $\gamma_{trk}(D)=(k+1)\gamma(D)$.
\end{remark}

\begin{theorem}\label{t1}
Let $k$ be a positive integer and let $D$ be a digraph of order $n\ge k$ with no isolated vertex. Then
$$\gamma_{t}(D)\le\gamma_{trk}(D)\le k\gamma_{t}(D).$$
Further, the following hold:
\begin{enumerate}
\item[$(a)$] The left equality holds if and only if there exists a $\gamma_{t}(D)$-set $X$ which can be partitioned into $k$ nonempty subsets
             $X_1,X_2,\dots,X_k$ such that for each $i\in\{1,2,\dots,k\}$, $V(D)\backslash X\subseteq N^+(X_i)$.
\item[$(b)$] The right equality holds if and only if there exists a $\gamma_{trk}(D)$-function $f$ such that for each $v\in V(D)$, either $f(v)=\{1,2,\dots,k\}$ or $f(v)=\emptyset$.
\end{enumerate}
\end{theorem}

\begin{proof} To prove the lower bound, let $f$ be a $\gamma_{trk}(D)$-function and let $V_0=\{v\in V(D):f(v)=\emptyset\}$.
We observe that $V(D)\backslash V_0$ is a TD-set of $D$ and hence $\gamma_{t}(D)\le|V(D)\backslash V_0|\le\omega(f)=\gamma_{trk}(D)$.

To prove the upper bound, let $X$ be a $\gamma_t(D)$-set.
It is easy to see that the function $f:V(D)\rightarrow\mathcal{P}(\{1,2,\ldots,k\})$ defined by $f(x)=\{1,2,\dots,k\}$ for each $x\in X$ and $f(x)=\emptyset$ for each $x\in V(D)\backslash X$, is a T$k$RDF on $D$ and hence $\gamma_{trk}(D)\le\omega(f)=k|X|=k\gamma_{t}(D)$.

(a) Suppose that $\gamma_{trk}(D)=\gamma_{t}(D)$.
Let $\gamma_{trk}(D)=\gamma_{t}(D)=n$. Then $X=V(D)$ is a unique $\gamma_{t}(D)$-set.
Since $n\ge k$, $X$ can be partitioned into $k$ nonempty subsets $X_1,X_2,\dots,X_k$.
We conclude that for each $i\in\{1,2,\dots,k\}$, $V(D)\backslash X=\emptyset\subseteq N^+(X_i)$.
So in the following we may assume that $\gamma_{trk}(D)=\gamma_{t}(D)<n$.
Let $f$ be a $\gamma_{trk}(D)$-function, $X_0=\{v\in V(D):f(v)=\emptyset\}$, $X_i=\{v\in V(D):f(v)=\{i\}\}$ for each $i\in\{1,2,\dots,k\}$ and let $X_{k+1}=V(D)\backslash \bigcup_{i=0}^kX_i$.
Since $\gamma_{trk}(D)<n$, $X_0\ne\emptyset$.
We observe that $X=\bigcup_{i=1}^{k+1}X_i$ is a TD-set of $D$ and hence
\begin{equation}\label{e1}
\gamma_{t}(D)\le\big|\bigcup_{i=1}^{k+1}X_i\big|\le\sum_{i=1}^{k+1}|X_i|\le\sum_{i=1}^{k}|X_i|+2|X_{k+1}|\le\omega(f)=\gamma_{trk}(D).
\end{equation}
Since $\gamma_{trk}(D)=\gamma_{t}(D)$, we have equality throughout the inequality chain (\ref{e1}).
 In particular, $|X_{k+1}|=0$, $X=\bigcup_{i=1}^{k}X_i$ and $\gamma_{trk}(D)=\gamma_{t}(D)=\sum_{i=1}^{k}|X_i|$.
Hence for each $i\in\{1,2,\dots,k\}$, every vertex in $X_0$ has at least one in-neighbor in $X_i$, and so
$X_i\ne\emptyset$ and $V(D)\backslash X=X_0\subseteq N^+(X_i)$.
Since $X_i\cap X_j=\emptyset$ for $1\le i<j\le k$ and $X=\bigcup_{i=1}^{k}X_i$,
we have that $\{X_1,X_2,\dots,X_k\}$ is a partition of $X$.
Moreover, since $X$ is a TD-set of $D$ and $\gamma_{t}(D)=\sum_{i=1}^{k}|X_i|=|X|$, $X$ is a $\gamma_t(D)$-set.

Conversely, suppose that there exists a $\gamma_{t}(D)$-set $X$ which can be partitioned into $k$ nonempty subsets $X_1,X_2,\dots,X_k$ such that for each $i\in\{1,2,\dots,k\}$, $V(D)\backslash X\subseteq N^+(X_i)$.
Observe that the function $g:V(D)\rightarrow\mathcal{P}(\{1,2,\ldots,k\})$ defined by
$g(x)=\{i\}$ for each $x\in X_i$ $(i\in\{1,2,\dots,k\})$ and $g(x)=\emptyset$ otherwise, is a T$k$RDF on $D$ and hence $\gamma_{trk}(D)\le\sum_{i=1}^{k}|X_i|=|X|=\gamma_{t}(D)$.
On the other hand, as proven earlier, $\gamma_{trk}(D)\ge\gamma_{t}(D)$.
Consequently, we have $\gamma_{trk}(D)=\gamma_{t}(D)$.

(b) Suppose that $\gamma_{trk}(D)=k\gamma_{t}(D)$. Let $X$ be a $\gamma_t(D)$-set.
It is easy to see that the function $f:V(D)\rightarrow\mathcal{P}(\{1,2,\ldots,k\})$ defined by $f(v)=\{1,2,\dots,k\}$ for each $v\in X$ and $f(v)=\emptyset$ for each $v\in V(D)\backslash X$, is a T$k$RDF on $D$ with weight $k\gamma_{t}(D)$, implying that $f$ is a $\gamma_{trk}(D)$-function such that for each $v\in V(D)$, either $f(v)=\{1,2,\dots,k\}$ or $f(v)=\emptyset$.

Conversely, suppose that there exists a $\gamma_{trk}(D)$-function $f$ such that for each $v\in V(D)$, either $f(v)=\{1,2,\dots,k\}$ or $f(v)=\emptyset$.
Let $X=\{v\in V(D):f(v)=\{1,2,\dots,k\}\}$.
Obviously $X$ is a TD-set of $D$ and so $k\gamma_{t}(D)\le k|X|=\gamma_{trk}(D)$.
On the other hand, as proven earlier, $\gamma_{trk}(D)\le k\gamma_{t}(D).$
As a result, we have $\gamma_{trk}(D)=k\gamma_{t}(D).$
\end{proof}

\begin{theorem}\label{t3}
Let $k$ be a positive integer and let $D$ be a connected digraph of order $n\ge \max\{k,2\}$.
Then $$\gamma_{rk}(D)\le\gamma_{trk}(D)\le 2\gamma_{rk}(D)-k+1$$
and these bounds are sharp.
\end{theorem}

\begin{proof} Since every T$k$RDF on $D$ is a $k$RDF on $D$, the lower bound holds.
To show the upper bound, let $f$ be a $\gamma_{rk}(D)$-function, $V_0=\{v\in V(D):f(v)=\emptyset\}$ and let $V(D)\backslash V_0=\{v_1,v_2,\dots,v_t\}$.
Since $\gamma_{trk}(D)\le n$, we may assume that $2\gamma_{rk}(D)-k+1<n$.
Note that $n\ge k$. Therefore, we have $\gamma_{rk}(D)<(n+k-1)/2\le(2n-1)/2$, which implies that $\gamma_{rk}(D)\le n-1$ and so $V_0\ne\emptyset$.
Let $X_1,X_2,\dots,X_r$ be the connected components of $D[\{v_1,v_2,\dots,v_t\}]$ and let $u$ be a vertex of $V_0$.
Since $u\in V_0$ must have in-neighbors in $\{v_1,v_2,\dots,v_t\}$, we may assume that $N^-(u)\cap X_i\ne\emptyset$ for each $1\le i\le s$.
Obviously, $s\le r$.
Moreover, since $\bigcup_{x\in N^-(u)}f(x)=\{1,2,\dots,k\}$,
we have $\sum_{i=1}^s\sum_{x\in X_i}|f(x)|\ge k$.
Note that $\sum_{i=s+1}^r\sum_{x\in X_i}|f(x)|\ge r-s$. Consequently, we have
$$\gamma_{rk}(D)=\sum_{i=1}^t|f(v_i)|=\sum_{i=1}^s\sum_{x\in X_i}|f(x)|+\sum_{i=s+1}^r\sum_{x\in X_i}|f(x)|\ge k+(r-s),$$
implying that $r-s\le\gamma_{rk}(D)-k$.
Since $D$ is connected, we can chose a vertex $u_i\in V_0$ for each $s+1\le i\le r$ such that $u_i$ has an in-neighbor or out-neighbor in $X_i$.
Then the function $g:V(D)\rightarrow\mathcal{P}(\{1,2,\ldots,k\})$ defined by
$g(u)=g(u_i)=\{1\}$ for each $s+1\le i\le r$ and $g(x)=f(x)$ otherwise, is a T$k$RDF on $D$ and hence
$$\gamma_{trk}(D)\le \omega(g)\le\omega(f)+r-s+1\le2\gamma_{rk}(D)-k+1,$$
establishing the desired upper bound.

To prove the sharpness of the lower bound, let $t\ge2$ be an integer and
let $D$ be the digraph obtained from the disjoint union of $t$ directed stars $S_{i_1}, S_{i_2}, \dots, S_{i_t}$, where $i_j\ge k+1$ for $1\le j\le t$,
by adding any number of arcs joining the $t$ centers of these directed stars so that the resulting digraph is connected.
One can check that the function $f:V(D)\rightarrow\mathcal{P}(\{1,2,\ldots,k\})$ that assigns the set $\{1,2,\dots,k\}$ to every support vertices of $D$ and the empty set to the remaining vertices of $D$, is a $\gamma_{rk}(D)$-function and is also a $\gamma_{trk}(D)$-function, implying that $\gamma_{trk}(D)=\gamma_{rk}(D)$.

To prove the sharpness of the upper bound, let $m$ be an arbitrary positive integer and let $D$ be a digraph with vertex set $V(D)=\{x_1,x_2,\dots,x_k,$ $y_1,y_2,\dots,y_m\}$ and arc set $A(D)=\{(x_i,y_j):1\le i\le k\ \text{and}\ 1\le j\le m\}$.
One can check that the function $f:V(D)\rightarrow\mathcal{P}(\{1,2,\ldots,k\})$ that assigns the set $\{i\}$ to $v_i$ for $1\le i\le k$ and the empty set to the remaining vertices of $D$, is a $\gamma_{rk}(D)$-function with weight $k$, and
the function $g:V(D)\rightarrow\mathcal{P}(\{1,2,\ldots,k\})$ that assigns the set $\{i\}$ to $v_i$ for $1\le i\le k$, $\{1\}$ to $y_1$ and the empty set to the remaining vertices of $D$, is a $\gamma_{trk}(D)$-function with weight $k+1$.
This implies that $\gamma_{trk}(D)=2\gamma_{rk}(D)-k+1$.
\end{proof}

\section{General bounds}

Our aim in this section is to present some sharp bounds on the total $k$-rainbow domination number.
We start with a simple but sharp lower and upper bounds on total $k$-rainbow domination number of a digraph.

\begin{proposition}\label{p1}
Let $k$ be a positive integer and let $D$ be a digraph of order $n$ with no isolated vertex. Then
$$\min\{k,n\}\le\gamma_{trk}(D)\le n$$
and these bounds are sharp for any digraph of order $n\le k$ with no isolated vertex.
\end{proposition}

\begin{proof} The upper bound is trivial. To show the lower bound, let $f$ be a $\gamma_{trk}(D)$-function. If $f(v)\ne\emptyset$ for each $v\in V(D)$, then clearly $\gamma_{trk}(D)=\omega(f)=n$. If there exists some vertex, say $v$, of $D$ such that $f(v)=\emptyset$, then $\bigcup_{x\in N^-(v)}f(x)=\{1,2,\dots,k\}$ and hence $\gamma_{trk}(D)=\omega(f)\ge\sum_{x\in N^-(v)}|f(x)|\ge|\bigcup_{x\in N^-(v)}f(x)|=k$, establishing the desired lower bound.
\end{proof}

We will provide a sufficient condition to have $\gamma_{trk}(D)=n$.
For this purpose, we first give a known result due to Hao and Qian \cite{Hao2}.

\begin{prelem}(\cite{Hao2})\label{thm}
Let $k$ be a positive integer and let $D$ be a digraph of order $n$ with $\Delta^-\geq\Delta^+$. If $k>(\Delta^-)^2$, then
$\gamma_{rk}(D)=n.$
\end{prelem}

As an immediate consequence of Theorems \ref{t3}, Theorem \ref{thm} and Proposition \ref{p1}, we have the
following corollary.

\begin{corollary}\label{c1}
Let $k$ be a positive integer and let $D$ be a digraph of order $n$ with $\Delta^-\geq\Delta^+\ge1$. If $k>(\Delta^-)^2$, then
$\gamma_{trk}(D)=n.$
\end{corollary}

Next result is an immediate consequence of Corollary \ref{c1}.

\begin{corollary}\label{c2}
Let $k\ge2$ be an integer. If $D$ is a directed path or a directed cycle of order $n$, then  $\gamma_{trk}(D)=n.$
\end{corollary}

Now we characterize all digraphs $D$ with $\gamma_{trk}(D)=k$.

\begin{theorem}\label{t4}
Let $k\ge2$ be an integer and let $D$ be a digraph of order $n$ with no isolated vertex.
Then $\gamma_{trk}(D)=k$ if and only if $D$ satisfies one of the following:
\begin{enumerate}
\item[$(a)$] $n=k$.
\item[$(b)$] $n\ge k+1$ and there exists a set $X=\{v_1,v_2,\dots,v_t\}\subseteq V(D)$, where $2\le t\le k$, such that the induced subdigraph $D[X]$ has no isolated vertex and $V(D)\backslash X\subseteq N^+(v_i)$ for $1\le i\le t$.
\end{enumerate}
\end{theorem}

\begin{proof}
Suppose that $\gamma_{trk}(D)=k$.
Let $f$ be a $\gamma_{trk}(D)$-function and let $V_0=\{v\in V(D):f(v)=\emptyset\}$.
If $V_0=\emptyset$, then clearly $n=\gamma_{trk}(D)=k$ and so (a) holds.
If $n<k$, then it follows from Proposition \ref{p1} that $\gamma_{trk}(D)=n<k$, a contradiction.
Hence we may assume that $V_0\ne\emptyset$ and $n\ge k+1$.
Let $v\in V_0$. Clearly, $\bigcup_{x\in N^-(v)}f(x)=\{1,2,\dots,k\}$.
Let $v_1,v_2,\dots,v_t$ be all vertices in $N^-(v)$ such that $f(v_i)\ne\emptyset$ for $1\le i\le t$.
Then $\sum_{i=1}^t|f(v_i)|\ge|\bigcup_{i=1}^tf(v_i)|=|\bigcup_{x\in N^-(v)}f(x)|=k$.
On the other hand, obviously $\sum_{i=1}^t|f(v_i)|\le\gamma_{trk}(D)=k$.
As a result, we have $\sum_{i=1}^t|f(v_i)|=k$.
Moreover, since $\bigcup_{i=1}^tf(v_i)=\bigcup_{x\in N^-(v)}f(x)=\{1,2,\dots,k\}$, we obtain that $\{f(v_1),f(v_2),\dots,f(v_t)\}$ is a partition of $\{1,2,\dots,k\}$ and so $V_0=V(D)\backslash \{v_1,v_2,\dots,v_t\}$.
This implies that $D[\{v_1,v_2,\dots,v_t\}]$ has no isolated vertex, $2\le t\le k$ and $V(D)\backslash \{v_1,v_2,\dots,v_t\}=V_0\subseteq N^+(v_i)$ for $1\le i\le t$.
Thus (b) holds.

Conversely, suppose that one of the two conditions (a) and (b) in
the statement of the theorem holds.
If (a) holds, that is, if $n=k$, then we conclude from Proposition \ref{p1} that $\gamma_{trk}(D)=k$.
Suppose, next, that (b) holds.
Observe that the function $g:V(D)\rightarrow\mathcal{P}(\{1,2,\ldots,k\})$ defined by $g(v_i)=\{i\}$ for $1\le i\le t-1$, $g(v_t)=\{t,t+1,\dots,k\}$ and $g(x)=\emptyset$ otherwise, is a T$k$RDF on $D$ and hence $\gamma_{trk}(D)\le k$.
On the other hand, it follows from Proposition \ref{p1} that $\gamma_{trk}(D)\ge\min\{k,n\}=k$.
As a result, we have $\gamma_{trk}(D)=k$, which completes our proof.
\end{proof}

\begin{theorem}\label{t6}
Let $k$ be a positive integer and let $D$ be a digraph of order $n$ with no isolated vertex.
Then $$\gamma_{trk}(D)\ge \left\lceil\frac{kn+1}{\Delta^++k}\right\rceil$$
and this lower bound is sharp.
\end{theorem}

\begin{proof} Let $f$ be a $\gamma_{trk}(D)$-function and let $V_i=\{v\in V(D): |f(v)|=i\}$ for each $i\in\{0,1,\ldots,k\}$.
Then $\gamma_{trk}(D)=|V_1|+2|V_2|+\dots+k|V_k|$ and $n=|V_0|+|V_1|+|V_2|+\dots+|V_k|$.
Let $A'=\{(u,v):u\in V(D)\backslash V_0\ \text{and}\ v\in V_0\}$.
Since $f$ is a $\gamma_{trk}(D)$-function, there must exist two distinct vertices $x,y\in V(D)\backslash V_0$ such that $(x,y)\in A(D)\backslash A'$ and so we obtain
\begin{eqnarray}
k|V_0|&\le&\sum_{(u,v)\in A'}|f(u)|  \nonumber\\
&\le&\Delta^+(|V_1|+2|V_2|+\dots+k|V_k|)-1  \nonumber\\
&=&\Delta^+\gamma_{trk}(D)-1.  \nonumber
\end{eqnarray}
Consequently, we have
\begin{eqnarray}
(\Delta^++k)\gamma_{trk}(D)&=&\Delta^+\gamma_{trk}(D)+k\gamma_{trk}(D)  \nonumber\\
&\ge&k|V_0|+k(|V_1|+2|V_2|+\dots+k|V_k|)+1  \nonumber\\
&=&k(|V_0|+|V_1|+|V_2|+\dots+|V_k|)+  \nonumber\\
   &&k(|V_2|+2|V_3|+\dots+(k-1)|V_k|)+1 \nonumber\\
&\ge&kn+1,  \nonumber
\end{eqnarray}
which implies the desired bound.

We next consider the sharpness of this inequality.
Let $k=1$ and let $P_3$ be a directed path of order $3$.
It is not hard to verify that $ \gamma_{trk}(P_3)=2$ and $\left\lceil\frac{k|V(P_3)|+1}{\Delta^++k}\right\rceil=\left\lceil\frac{4}{2}\right\rceil=2$, implying that $\gamma_{trk}(P_3)=\left\lceil\frac{k|V(P_3)|+1}{\Delta^++k}\right\rceil$.
Let $k\ge2$ be an integer and let $D$ be a digraph with vertex set $V(D)=\{u_1,u_2,\dots,u_k,$ $v_1,v_2,\dots,v_k\}$ and arc set $A(D)=\{(v_i,u_j):1\le i,j\le k\}\cup\{(v_i,v_{i+1}):1\le i\le k-1\}$.
Let $X=\{v_1,v_2,\dots,v_k\}$.
It follows from Theorem \ref{t4} (b) that $\gamma_{trk}(D)=k$.
Moreover, since $|V(D)|=2k$ and $\Delta^+=k+1$,
$\left\lceil\frac{k|V(D)|+1}{\Delta^++k}\right\rceil=\left\lceil\frac{2k^2+1}{2k+1}\right\rceil=k.$
As a result, we have $\gamma_{trk}(D)=\left\lceil\frac{k|V(D)|+1}{\Delta^++k}\right\rceil$.
\end{proof}

\begin{theorem}\label{t7}
For any digraph $D$ with no isolated vertex and two positive integers $k$ and $k'$ with $k'>k$,
$$\gamma_{trk'}(D)\leq\gamma_{trk}(D)+(k'-k)\left\lfloor\frac{\gamma_{trk}(D)}{k}\right\rfloor$$
and this upper bound is sharp.
\end{theorem}

\begin{proof}
Let $f$ be a $\gamma_{trk}(D)$-function. For each $i\in\{1,2,\ldots,k\}$, let $V_i=\{v\in V(D): i\in f(v)\}$.
It follows that $\gamma_{trk}(D)=\omega(f)=\sum_{i=1}^k|V_i|$.
Without loss of generality, we may assume that $|V_1|\geq |V_2|\geq\cdots\geq |V_k|$. Clearly $|V_k|\leq \lfloor\frac{\gamma_{trk}(D)}{k}\rfloor$.
Then the function $g:V(D)\rightarrow\mathcal{P}(\{1,2,\ldots,k\})$ defined by
\begin{displaymath}
g(v)\\ =\left\{\begin{array}{ll}
f(v)\cup\{k+1,k+2,\ldots,k'\},&\textrm{if}\ v\in V_k,\\
f(v),&\textrm{otherwise},
\end{array}\right.
 \end{displaymath}
is a T$k'$RDF on $D$ and hence
\begin{align}
\gamma_{trk'}(D)\leq& \omega(g)\nonumber\\
=&\sum_{v\in V_k}|f(v)\cup\{k+1,k+2,\ldots,k'\}|+\sum_{v\in V(D)\backslash V_k}|f(v)|\nonumber\\
=&(k'-k)|V_k|+\sum_{v\in V(D)}|f(v)|\nonumber\\
\leq&\gamma_{trk}(D)+(k'-k)\left\lfloor\frac{\gamma_{trk}(D)}{k}\right\rfloor,\nonumber
\end{align}
establishing the desired upper bound.

To prove the sharpness,
let $t\ge2$ and let $D$ be the digraph obtained from the disjoint union of $t$ directed stars $S_{i_1}, S_{i_2}, \dots, S_{i_t}$, where $i_j\ge k'+1$ for $1\le j\le t$,
by adding any number of arcs joining the $t$ centers of these directed stars so that the subdigraph induced by these $t$ centers has no isolated vertex.
One can check that the function $f:V(D)\rightarrow\mathcal{P}(\{1,2,\ldots,k\})$ that assigns the set $\{1,2,\dots,k\}$ to every support vertices of $D$ and the empty set to the remaining vertices of $D$, is a unique $\gamma_{trk}(D)$-function with weight $kt$; and the function $g:V(D)\rightarrow\mathcal{P}(\{1,2,\ldots,k'\})$ that assigns the set $\{1,2,\dots,k'\}$ to every support vertices of $D$ and the empty set to the remaining vertices of $D$, is a unique $\gamma_{trk'}(D)$-function with weight $k't$.
Therefore, we obtain $$\gamma_{trk'}(D)=k't=\gamma_{trk}(D)+(k'-k)\left\lfloor\frac{\gamma_{trk}(D)}{k}\right\rfloor,$$
 which completes our proof.
\end{proof}

\section{Cartesian product of directed paths}

Let $P_n$ denote the directed path of order $n$ with vertex set $V(P_n)=\{0, 1, 2, \ldots, n-1\}$ and
arc set $A(P_n)=\{(i, i+1):i=0, 1, 2, \ldots, n-2\}$.

In this section, we shall determine the exact values of $\gamma_{tr2}(P_2\square P_n)$, $\gamma_{tr3}(P_2\square P_n)$ and $\gamma_{tr3}(P_3\square P_n)$.
Now we consider the exact value of total $2$-rainbow domination number of the Cartesian product $P_2\square P_n$.
To our aim, the following lemmas are essential.

\begin{lemma}\label{l1}
Let $n\ge2$ be an integer and let $f$ be a $\gamma_{tr2}(P_2\square P_n)$-function such that the number of vertices assigned $\emptyset$ under $f$ is minimum.
Then $|f((0,0))|+|f((1,0))|\ge2$ and $|f((0,j))|+|f((1,j))|\ge1$ for each $j\in \{1,2,\dots,n-1\}$.
\end{lemma}

\begin{proof} Since $d^-((0,0))=0$, $|f((0,0))|\ge1$.
If $|f((0,0))|=2$, then $|f((0,0))|+|f((1,0))|\ge2$ and
if $|f((0,0))|=1$, then clearly $|f((1,0))|\ge1$ and so $|f((0,0))|+|f((1,0))|\ge2$.
We now prove that $|f((0,j))|+|f((1,j))|\ge1$ for each $j\in \{1,2,\dots,n-1\}$.
Suppose, to the contrary, that there exists some $j_0\in \{1,2,\dots,n-1\}$ such that $|f((0,j_0))|+|f((1,j_0))|=0$.
Clearly $f((0,j_0-1))=f((1,j_0-1))=\{1,2\}$.
Then the function $g:V(D)\rightarrow\mathcal{P}(\{1,2\})$ defined by
$g((0,j_0-1))=g((0,j_0))=g((1,j_0-1))=g((1,j_0))=\{1\}$ and $g((i,j))=f((i,j))$ otherwise, is a T$2$RDF on $P_2\square P_n$ with weight $\omega(f)$ and so $g$ is also a $\gamma_{tr2}(P_2\square P_n)$-function, a contradiction to the choice of $f$.
Thus $|f((0,j))|+|f((1,j))|\ge1$ for each $j\in \{1,2,\dots,n-1\}$.
\end{proof}

\begin{lemma}\label{l2}
Let $n\ge2$ be an integer, $f$ be a $\gamma_{tr2}(P_2\square P_n)$-function such that the number of vertices assigned $\emptyset$ under $f$ is minimum and let $a_j=|f((0,j))|+|f((1,j))|$ for each $j\in \{0,1,\dots,n-1\}$.
Then for each $j\in \{0,1,\dots,n-2\}$, $$a_j+a_{j+1}\ge3.$$
\end{lemma}

\begin{proof} It follows from Lemma \ref{l1} that $a_0+a_1\ge3.$
We now claim that $a_j+a_{j+1}\ge3$ for each $j\in \{1,2,\dots,n-2\}$.
Suppose, to the contrary, that there exists some $j_0\in \{1,2,\dots,n-2\}$ such that $a_{j_0}+a_{j_0+1}\le2$.
By Lemma \ref{l1}, we have that $a_{j_0}\ge1$ and $a_{j_0+1}\ge1$ and so $a_{j_0}+a_{j_0+1}=2$.
This implies that $a_{j_0}=a_{j_0+1}=1$.
If $|f((0,j_0+1))|=0$ and $|f((1,j_0+1))|=1$, then $f((0,j_0))=\{1,2\}$ and so $a_{j_0}\ge2$, a contradiction.
Hence we may assume that $|f((0,j_0+1))|=1$ and $|f((1,j_0+1))|=0$.
Then $\{1,2\}\backslash f((0,j_0+1))\subseteq f((1,j_0))$, implying that $|f((1,j_0))|\ge1$.
Moreover, since $a_{j_0}=|f((0,j_0))|+|f((1,j_0))|=1$, we have $|f((1,j_0))|=1$ and $|f((0,j_0))|=0$.
Clearly, $f((0,j_0-1))=\{1,2\}$.
Since $|f((0,j_0))|=|f((1,j_0+1))|=0$ and $|f((1,j_0))|=1$, we conclude from the definition of $\gamma_{tr2}(P_2\square P_n)$-function that $|f((1,j_0-1))|\ge1$.
Then the function $g:V(P_2\square P_n)\rightarrow\mathcal{P}(\{1,2\})$ defined by
$g((0,j_0-1))=g((0,j_0))=\{1\}$ and $g((i,j))=f((i,j))$ otherwise, is also a $\gamma_{tr2}(P_2\square P_n)$-function, a contradiction to the choice of $f$.
Consequently, $a_j+a_{j+1}\ge3$ for each $j\in \{1,2,\dots,n-2\}$.
\end{proof}

\begin{proposition}\label{p2}
For $n\ge2$,
$$\gamma_{tr2}(P_2\square P_n)=\left\lceil\frac{3n}{2}\right\rceil.$$
\end{proposition}

\begin{proof} Let $f$ be a $\gamma_{tr2}(P_2\square P_n)$-function such that the number of vertices assigned $\emptyset$ under $f$ is minimum
and let $a_j=|f((0,j))|+|f((1,j))|$ for each $j\in \{0,1,\dots,n-1\}$.
By Lemmas \ref{l1} and \ref{l2}, we have that $a_0\ge2$ and $a_j+a_{j+1}\ge3$ for each $j\in \{0,1,\dots,n-2\}$. Therefore, if $n$ is odd, then
$$\gamma_{tr2}(P_2\square P_n)=\omega(f)=a_0+\sum_{j=1}^{(n-1)/2}(a_{2j-1}+a_{2j})\ge 2+\frac{3(n-1)}{2}=\left\lceil\frac{3n}{2}\right\rceil,$$
and if $n$ is even, then
$$\gamma_{tr2}(P_2\square P_n)=\omega(f)=\sum_{j=0}^{(n-2)/2}(a_{2j}+a_{2j+1})\ge\frac{3n}{2}=\left\lceil\frac{3n}{2}\right\rceil.$$

To show the upper bound, we now provide a T$2$RDF $g:V(P_2\square P_n)\rightarrow\mathcal{P}(\{1,2\})$ defined by
\begin{displaymath}
g((i,j))\\ =\left\{\begin{array}{ll}
\{1\},&\textrm{if}\ i=0\ \textrm{and}\ j\in \{0,1,\dots,n-1\},\\
\{2\},&\textrm{if}\ i=1\ \textrm{and}\ j\in \{0,1,\dots,n-1\}\ \textrm{is even} ,\\
\emptyset,&\textrm{otherwise},
\end{array}\right.
 \end{displaymath}
and so
$$\gamma_{tr2}(P_2\square P_n)\le\omega(g)=n+\left\lceil\frac{n}{2}\right\rceil=\left\lceil\frac{3n}{2}\right\rceil,$$
which completes our proof.
\end{proof}

\begin{lemma}\label{l3}
Let $n\ge2$ be an integer, $f$ be a $\gamma_{tr3}(P_2\square P_n)$-function such that the number of vertices assigned $\emptyset$ under $f$ is minimum and let $a_j=|f((0,j))|+|f((1,j))|$ for each $j\in \{0,1,\dots,n-1\}$.
Then for each $j\in \{0,1,\dots,n-1\}$, $$a_j\ge2.$$
\end{lemma}

\begin{proof}
Suppose, to the contrary, that there exists some $j_0\in \{0,1,\dots,n-1\}$ such that $a_{j_0}\le1$.
Assume that $a_{j_0}=0$. Clearly, $f((0,j_0-1))=f((1,j_0-1))=\{1,2,3\}$.
Then the function $g:V(P_2\square P_n)\rightarrow\mathcal{P}(\{1,2,3\})$ defined by
$g((0,j_0-1))=g((0,j_0))=g((1,j_0-1))=g((1,j_0))=\{1\}$ and $g((i,j))=f((i,j))$ otherwise, is a T$3$RDF on $P_2\square P_n$ with weight $\omega(g)=\omega(f)-2$, a contradiction.

Assume, next, that $a_{j_0}=1$.
Suppose that $|f((0,j_0))|=1$ and $|f((1,j_0))|=0$. Then $\{1,2,3\}\backslash f((0,j_0))\subseteq f((1,j_0-1))$,
implying that $|f((1,j_0-1))|\in\{2,3\}$.
One can check that the function $g:V(P_2\square P_n)\rightarrow\mathcal{P}(\{1,2,3\})$ defined by
$g((1,j_0-1))=g((1,j_0))=\{1\}$ and $g((i,j))=f((i,j))$ otherwise, is a T$3$RDF on $P_2\square P_n$.
Thus, if $|f((1,j_0-1))|=2$, then $g$ is also a $\gamma_{tr3}(P_2\square P_n)$-function, a contradiction to the choice of $f$;
and if $|f((1,j_0-1))|=3$, then $g$ is a T$3$RDF on $P_2\square P_n$ with weight $\omega(f)-1$, a contradiction.
Suppose now that $|f((0,j_0))|=0$ and $|f((1,j_0))|=1$.
Clearly, $f((0,j_0-1))=\{1,2,3\}$.
If $|f((1,j_0-1))|=0$, then the function $g:V(P_2\square P_n)\rightarrow\mathcal{P}(\{1,2,3\})$ defined by
$g((0,j_0-1))=g((0,j_0))=g((1,j_0-1))=\{1\}$ and $g((i,j))=f((i,j))$ otherwise, is a T$3$RDF on $P_2\square P_n$ with weight $\omega(f)$
and so $g$ is also a $\gamma_{tr3}(P_2\square P_n)$-function, a contradiction to the choice of $f$.
If $|f((1,j_0-1))|\ge1$, then the function $g:V(P_2\square P_n)\rightarrow\mathcal{P}(\{1,2,3\})$ defined by
$g((0,j_0-1))=g((0,j_0))=\{1\}$ and $g((i,j))=f((i,j))$ otherwise, is a T$3$RDF on $P_2\square P_n$ with weight $\omega(f)-1$, a contradiction.
Consequently, we have $a_j\ge2$ for each $j\in \{0,1,\dots,n-1\}$.
\end{proof}

Next we shall give the exact value of $\gamma_{tr3}(P_2\square P_n)$.

\begin{proposition}\label{p3}
For $n\ge2$,
$$\gamma_{tr3}(P_2\square P_n)=2n.$$
\end{proposition}

\begin{proof} Let $f$ be a $\gamma_{tr3}(P_2\square P_n)$-function such that the number of vertices assigned $\emptyset$ under $f$ is minimum and let $a_j=|f((0,j))|+|f((1,j))|$ for each $j\in \{0,1,\dots,n-1\}$.
We conclude from Lemma \ref{l3} that for each $j\in \{0,1,\dots,n-1\}$, $a_j\ge2.$
This implies that $\gamma_{tr3}(P_2\square P_n)=\omega(f)\ge2n.$
On the other hand, it follows from Proposition \ref{p1} that $\gamma_{tr3}(P_2\square P_n)\le2n.$
As a result, we obtain $\gamma_{tr3}(P_2\square P_n)=2n.$
\end{proof}

For any integer $n\ge3$, we next determine the value of $\gamma_{tr3}(P_3\square P_n)$.
For this purpose, we need some lemmas as follows.

\begin{lemma}\label{l4}
Let $n\ge3$ be an integer and let $f$ be a $\gamma_{tr3}(P_3\square P_n)$-function such that the number of vertices assigned $\emptyset$ under $f$ is minimum.
Then for each $j\in \{0,1,\dots,n-1\}$, $$|f((2,j))|\in\{0,1\}.$$
\end{lemma}

\begin{proof} It is not hard to verify that $|f((2,n-1))|\in\{0,1\}$.
We now show that $|f((2,j))|\in\{0,1\}$ for each $j\in \{0,1,\dots,n-2\}$.
Suppose, to the contrary, that there exists some $j_0\in \{0,1,\dots,n-2\}$ such that $|f((2,j_0))|\ge2$.
If $|f((2,j_0+1))|\ge1$, then the function $g_1:V(P_3\square P_n)\rightarrow\mathcal{P}(\{1,2,3\})$ defined by
$g_1((2,j_0))=\{1\}$ and $g_1((i,j))=f((i,j))$ otherwise, is a T$3$RDF on $P_3\square P_n$ with weight $\omega(g_1)\le\gamma_{tr3}(P_3\square P_n)-1$, a contradiction.
Hence we may assume that $|f((2,j_0+1))|=0$.
Now define the function $g_2:V(P_3\square P_n)\rightarrow\mathcal{P}(\{1,2,3\})$ by
$g_2((2,j_0))=g_2((2,j_0+1))=\{1\}$ and $g_2((i,j))=f((i,j))$ otherwise.
If $|f((2,j_0))|=2$, then $g_2$ is a T$3$RDF on $P_3\square P_n$ with weight $\omega(g_2)=\gamma_{tr3}(P_3\square P_n)$ and hence $g_2$ is
also a $\gamma_{tr3}(P_3\square P_n)$-function, a contradiction to the
choice of $f$.
If $|f((2,j_0))|=3$, then $g_2$ is a T$3$RDF on $P_3\square P_n$ with weight $\omega(g_2)=\gamma_{tr3}(P_3\square P_n)-1$, a contradiction.
Consequently, we have $|f((2,j))|\in\{0,1\}$ for each $j\in \{0,1,\dots,n-2\}$.
\end{proof}

\begin{lemma}\label{l5}
Let $n\ge3$ be an integer and let $f$ be a $\gamma_{tr3}(P_3\square P_n)$-function such that the number of vertices assigned $\emptyset$ under $f$ is minimum. If there exists $j\in \{0,1,\dots,n-1\}$ such that $|f((1,j))|=1$, then $|f((2,j))|=1$.
\end{lemma}

\begin{proof}
Suppose that there exists $j\in \{0,1,\dots,n-1\}$ such that $|f((1,j))|=1$.
If $|f((2,j))|=0$, then $\{1,2,3\}\backslash f((1,j))\subseteq f((2,j-1))$ and so $|f((2,j-1))|\ge2$, a contradiction to Lemma \ref{l4}.
Thus $|f((2,j))|\ge1$.
Moreover, since $|f((2,j))|\in\{0,1\}$ by Lemma \ref{l4}, we have $|f((2,j))|=1$.
\end{proof}

\begin{lemma}\label{l6}
Let $n\ge3$ be an integer and let $f$ be a $\gamma_{tr3}(P_3\square P_n)$-function such that the number of vertices assigned $\emptyset$ under $f$ is minimum.
Then for $i\in \{0,1\}$ and $j\in \{0,1,\dots,n-1\}$,
$$|f((i,j))|\le2.$$
\end{lemma}

\begin{proof} If $|f((0,n-1))|=3$, then the function $g:V(P_3\square P_n)\rightarrow\mathcal{P}(\{1,2,3\})$ defined by
$g((0,n-1))=\{1\}$, $g((1,n-1))=f((1,n-1))\cup \{1\}$ and $g((i,j))=f((i,j))$ otherwise, is a T$3$RDF on $P_3\square P_n$ with weight $\omega(g)\le\gamma_{tr3}(P_3\square P_n)-1$, a contradiction.
Thus $|f((0,n-1))|\le2$. Similarly, we have $|f((1,n-1))|\le2$.

We now claim that $|f((i,j))|\le2$ for $i\in \{0,1\}$ and $j\in \{0,1,\dots,n-2\}$.
Suppose, to the contrary, that there exist $i_0\in \{0,1\}$ and $j_0\in \{0,1,\dots,n-2\}$ such that $|f((i_0,j_0))|=3$.
Assume that $|f((i_0,j_0+1))|\ge1$ or $|f((i_0+1,j_0))|\ge1$. Without loss of generality, we may assume that $|f((i_0,j_0+1))|\ge1$.
Then the function $g:V(P_3\square P_n)\rightarrow\mathcal{P}(\{1,2,3\})$ defined by
$g((i_0,j_0))=\{1\}$, $g((i_0+1,j_0))=f((i_0+1,j_0))\cup\{1\}$ and $g((i,j))=f((i,j))$ otherwise, is a T$3$RDF on $P_3\square P_n$ with weight $\omega(g)\le\omega(f)-1$, a contradiction.
Assume now that $|f((i_0,j_0+1))|=|f((i_0+1,j_0))|=0$.
Then the function $g:V(P_3\square P_n)\rightarrow\mathcal{P}(\{1,2,3\})$ defined by
$g((i_0,j_0))=g((i_0+1,j_0))=g((i_0,j_0+1))=\{1\}$ and $g((i,j))=f((i,j))$ otherwise, is a T$3$RDF on $P_3\square P_n$ with weight $\omega(g)=\omega(f)$, implying that $g$ is also a $\gamma_{tr3}(P_3\square P_n)$-function, a contradiction to the
choice of $f$.
Therefore, we have $|f((i,j))|\le2$ for $i\in \{0,1\}$ and $j\in \{0,1,\dots,n-2\}$.
\end{proof}

\begin{lemma}\label{l7}
Let $n\ge3$ be an integer and let $f$ be a $\gamma_{tr3}(P_3\square P_n)$-function such that the number of vertices assigned $\emptyset$ under $f$ is minimum. If there exists $j\in \{0,1,\dots,n-2\}$ such that $|f((1,j))|=2$, then $|f((1,j+1))|=|f((2,j))|=0$.
\end{lemma}

\begin{proof} Suppose that there exists $j_0\in \{0,1,\dots,n-2\}$ such that $|f((1,j_0))|=2$.
Assume that $|f((1,j_0+1))|\ge1$ and $|f((2,j_0))|\ge1$. Then the function $g:V(P_3\square P_n)\rightarrow\mathcal{P}(\{1,2,3\})$ defined by
$g((1,j_0))=\{1\}$ and $g((i,j))=f((i,j))$ otherwise, is a T$3$RDF on $P_3\square P_n$ with weight $\omega(g)=\omega(f)-1$, a contradiction.
Assume that exactly one of $|f((1,j_0+1))|$ and $|f((2,j_0))|$ is equal to $0$.
Without loss of generality, assume that $|f((1,j_0+1))|=0$ and $|f((2,j_0))|\ge1$.
Then the function $g:V(P_3\square P_n)\rightarrow\mathcal{P}(\{1,2,3\})$ defined by
$g((1,j_0))=g((1,j_0+1))=\{1\}$ and $g((i,j))=f((i,j))$ otherwise, is a T$3$RDF on $P_3\square P_n$ with weight $\omega(g)=\omega(f)$ and so
$g$ is also a $\gamma_{tr3}(P_3\square P_n)$-function, a contradiction to the choice of $f$.
Therefore, we have $|f((1,j_0+1))|=|f((2,j_0))|=0$, which completes our proof.
\end{proof}

\begin{lemma}\label{l8}
Let $n\ge3$ be an integer and let $f$ be a $\gamma_{tr3}(P_3\square P_n)$-function such that the number of vertices assigned $\emptyset$ under $f$ is minimum. Then $|f((i,0))|=|f((0,j))|=1$ for each $i\in \{0,1,2\}$ and $j\in \{1,2,\dots,n-1\}$.
\end{lemma}

\begin{proof} By Lemma \ref{l6}, $|f((0,0))|\le2$. If $|f((0,0))|=2$, then $|f((1,0))|\ge1$ and $|f((0,1))|\ge1$ and so the function $f':V(P_3\square P_n)\rightarrow\mathcal{P}(\{1,2,3\})$ defined by
$f'((0,0))=\{1\}$ and $f'((i,j))=f((i,j))$ otherwise, is a T$3$RDF on $P_3\square P_n$ with weight $\omega(f')=\omega(f)-1$, a contradiction.
Thus $|f((0,0))|\le1$.
Moreover, since $d^-((0,0))=0$, we obtain $|f((0,0))|=1$.

We now claim that $|f((0,j))|=1$ for each $j\in \{1,2,\dots,n-1\}$.
Note that $|f((0,j))|\le2$ for $j\in \{1,2,\dots,n-1\}$ by Lemma \ref{l6}.
For the sake of contradiction, we may assume that there exists $j_0\in \{1,2,\dots,n-1\}$ such that $|f((0,j_0))|\in\{0,2\}$.
If $|f((0,j_0))|=0$, then clearly $f((0,j_0-1))=\{1,2,3\}$, a contradiction to Lemma \ref{l6}.
Assume, next, that $|f((0,j_0))|=2$.
Obviously $|f((0,j_0+1))|\ge1$.
If $|f((1,j_0))|\ge1$, then the function $g:V(P_3\square P_n)\rightarrow\mathcal{P}(\{1,2,3\})$ defined by
$g((0,j_0))=\{1\}$ and $g((i,j))=f((i,j))$ otherwise, is a T$3$RDF on $P_3\square P_n$ with weight $\omega(g)=\omega(f)-1$, a contradiction.
If $|f((1,j_0))|=0$, then the function $h:V(P_3\square P_n)\rightarrow\mathcal{P}(\{1,2,3\})$ defined by
$h((0,j_0))=h((1,j_0))=\{1\}$ and $h((i,j))=f((i,j))$ otherwise, is a T$3$RDF on $P_3\square P_n$ with weight $\omega(h)=\omega(f)$ and so $h$
is also a $\gamma_{tr3}(P_3\square P_n)$-function, a contradiction to the choice of $f$.
Consequently, $|f((0,j))|=1$ for each $j\in \{1,2,\dots,n-1\}$.
Similarly, we have $|f((i,0))|=1$ for each $i\in \{1,2\}$.
\end{proof}

\begin{lemma}\label{l9}
Let $n\ge3$ be an integer and let $f$ be a $\gamma_{tr3}(P_3\square P_n)$-function such that the number of vertices assigned $\emptyset$ under $f$ is minimum. If there exists $j\in \{0,1,\dots,n-2\}$ such that $|f((1,j))|=|f((2,j))|=1$, then $|f((1,j+1))|+|f((2,j+1))|=2$.
\end{lemma}

\begin{proof} Suppose that there exists $j_0\in \{0,1,\dots,n-2\}$ such that $|f((1,j_0))|=|f((2,j_0))|=1$.
Note that $|f((0,j_0+1))|=1$ by Lemma \ref{l8}. Thus $|f((1,j_0+1))|\ge1$ and hence it follows from Lemma \ref{l6} that $|f((1,j_0+1))|\in\{1,2\}$.

Assume now that $|f((1,j_0+1))|=1$. Recall that $|f((2,j_0))|=1$. Then $|f((2,j_0+1))|\ge1$.
Moreover, it follows from Lemma \ref{l4} that $|f((2,j_0+1))|\in\{0,1\}$, implying that $|f((2,j_0+1))|=1$.
As a result, we get $|f((1,j_0+1))|+|f((2,j_0+1))|=2$.

Assume that $|f((1,j_0+1))|=2$.
If $j_0+1<n-1$, then we conclude from Lemma \ref{l7} that $|f((2,j_0+1))|=0$ and so $|f((1,j_0+1))|+|f((2,j_0+1))|=2$.
Now let $j_0+1=n-1$.
If $|f((2,j_0+1))|\ne0$, then the function $g:V(P_3\square P_n)\rightarrow\mathcal{P}(\{1,2,3\})$ defined by
$g((1,j_0+1))=\{1\}$ and $g((i,j))=f((i,j))$ otherwise, is a T$3$RDF on $P_3\square P_n$ with weight $\omega(g)=\omega(f)-1$, a contradiction.
If $|f((2,j_0+1))|=0$, then the function $g:V(P_3\square P_n)\rightarrow\mathcal{P}(\{1,2,3\})$ defined by
$g((1,j_0+1))=g((2,j_0+1))=\{1\}$ and $g((i,j))=f((i,j))$ otherwise, is a T$3$RDF on $P_3\square P_n$ with weight $\omega(g)=\omega(f)$ and so $g$
is also a $\gamma_{tr3}(P_3\square P_n)$-function, a contradiction to the choice of $f$.

The proof is completed.
\end{proof}

\begin{proposition}\label{p4}
For $n\ge3$,
\begin{displaymath}
\gamma_{tr3}(P_3\square P_n)\\ =\left\{\begin{array}{ll}
\left\lceil\frac{8n}{3}\right\rceil+1,&\textrm{if}\ n\equiv0\pmod3,\\
\left\lceil\frac{8n}{3}\right\rceil,&\textrm{if}\ n\equiv1,2\pmod3.
\end{array}\right.
 \end{displaymath}
\end{proposition}

\begin{proof} Let $f$ be a $\gamma_{tr3}(P_3\square P_n)$-function such that the number of vertices assigned $\emptyset$ under $f$ is minimum
and let $a_j=\sum_{i=0}^2|f((i,j))|$ for each $j\in \{0,1,\dots,n-1\}$.
Let $j\in \{0,1,\dots,n-3\}$. We now claim that $a_j+a_{j+1}+a_{j+2}\ge8$.
Note that $|f((1,j+1))|\in\{0,1,2\}$ by Lemma \ref{l6}.
We now consider three cases.

\vspace*{0.5mm}
\textbf{Case 1.} Suppose that $|f((1,j+1))|=0$.
\vspace*{0.5mm}

If $|f((2,j+1))|=0$, then clearly $f((2,j))=\{1,2,3\}$, a contradiction to Lemma \ref{l4}.
Hence by Lemma \ref{l4}, we have $|f((2,j+1))|=1$.
Since $|f((0,j+1))|=1$ by Lemma \ref{l8} and $|f((1,j+1))|=0$, we obtain $\{1,2,3\}\backslash f((0,j+1))\subseteq f((1,j))$ and so by Lemma \ref{l6}, $|f((1,j))|=2$.
Then it follows from Lemma \ref{l7} that $|f((2,j))|=0$.
Moreover, since the subdigraph of $P_3\square P_n$ induced by the set $\{v\in V(P_3\square P_n):f(v)\ne\emptyset\}$ has no isolated vertex, $|f((2,j+2))|\ge1$.
Recall that $|f((0,j+2))|=1$ by Lemma \ref{l8} and $|f((1,j+1))|=0$. Thus $|f((1,j+2))|\ge1$.
Note that $|f((0,j))|=|f((0,j+1))|=|f((0,j+2))|=1$ by Lemma \ref{l8}.
Therefore, we conclude that $a_j+a_{j+1}+a_{j+2}\ge8$.

\vspace*{0.5mm}
\textbf{Case 2.} Suppose that $|f((1,j+1))|=1$.
\vspace*{0.5mm}

If $|f((1,j))|=|f((2,j))|=0$, then $f((2,j-1))=\{1,2,3\}$, a contradiction to Lemma \ref{l4}.
Thus we have $|f((1,j))|+|f((2,j))|\ge1$.
Since $|f((1,j+1))|=1$, we conclude from Lemma \ref{l5} that $|f((2,j+1))|=1$.
And it follows from Lemma \ref{l9} that $|f((1,j+2))|+|f((2,j+2))|=2$.
Note that $|f((0,j))|=|f((0,j+1))|=|f((0,j+2))|=1$ by Lemma \ref{l8}.
Therefore, we obtain $a_j+a_{j+1}+a_{j+2}\ge8$.

\vspace*{0.5mm}
\textbf{Case 3.} Suppose that $|f((1,j+1))|=2$.
\vspace*{0.5mm}

It follows from Lemma \ref{l7} that $|f((1,j+2))|=|f((2,j+1))|=0$, implying that $|f((2,j))|\ge1$ and $|f((2,j+2))|\ge1$ and so by Lemma \ref{l4}, we obtain $|f((2,j))|=|f((2,j+2))|=1$.
Now suppose that $|f((1,j))|=0$.
Moreover, since $|f((0,j))|=1$ by Lemma \ref{l8}, $\{1,2,3\}\backslash f((0,j))\subseteq f((1,j-1))$ and so by Lemma \ref{l6}, $|f((1,j-1))|=2$.
We conclude from Lemma \ref{l7} that $|f((2,j-1))|=0$.
Note that $|f((2,j-1))|=|f((1,j))|=|f((2,j+1))|=0$ and $|f((2,j))|=1$. This is a contradiction to the fact that the subdigraph of $P_3\square P_n$ induced by the set $\{v\in V(P_3\square P_n):f(v)\ne\emptyset\}$ has no isolated vertex.
Therefore, we have $|f((1,j))|\ge1$.
Recall that $|f((0,j))|=|f((0,j+1))|=|f((0,j+2))|=1$ by Lemma \ref{l8}.
Therefore, we obtain $a_j+a_{j+1}+a_{j+2}\ge8$.

\vspace*{0.5mm}

By the above arguments, we have that for each $j\in \{0,1,\dots,n-3\}$, $a_j+a_{j+1}+a_{j+2}\ge8$.
Since $|f((0,0))|=|f((1,0))|=|f((2,0))|=|f((0,1))|=1$ by Lemma \ref{l8}, we have that $a_0=3$ and it follows from Lemma \ref{l9} that $|f((1,1))|+|f((2,1))|=2$, implying that $a_1=3$.
Let $n=3k+l\ge3$, where $k$ is a positive integer and $l\in\{0,1,2\}$.
If $l=0$, then
\begin{equation}\label{e123456}
\gamma_{tr3}(P_3\square P_n)=\omega(f)=\sum_{i=0}^{k-1}(a_{3i}+a_{3i+1}+a_{3i+2})\ge 8k=\left\lceil\frac{8n}{3}\right\rceil,
\end{equation}
if $l=1$, then
$$\gamma_{tr3}(P_3\square P_n)=\omega(f)=a_0+\sum_{i=0}^{k-1}(a_{3i+1}+a_{3i+2}+a_{3i+3})\ge 8k+3=\left\lceil\frac{8n}{3}\right\rceil,$$
and if $l=2$, then
$$\gamma_{tr3}(P_3\square P_n)=\omega(f)=a_0+a_1+\sum_{i=1}^{k}(a_{3i-1}+a_{3i}+a_{3i+1})\ge 8k+6=\left\lceil\frac{8n}{3}\right\rceil.$$

If fact, if $l=0$, then $\gamma_{tr3}(P_3\square P_n)\ge \left\lceil\frac{8n}{3}\right\rceil+1.$
Suppose, to the contrary, that $\gamma_{tr3}(P_3\square P_n)=\left\lceil\frac{8n}{3}\right\rceil$.
Then we have equality throughout the inequality chain (\ref{e123456}), implying that $a_{3i}+a_{3i+1}+a_{3i+2}=8$ for each $i\in \{0,1,\dots,n/3-1\}$.

Recall that $|f((0,0))|=|f((1,0))|=|f((2,0))|=|f((0,1))|=|f((0,2))|=1$ by Lemma \ref{l8}.
Suppose that $|f((1,1))|=1$.
Moreover, since $|f((2,0))|=1$, we have $|f((2,1))|\ge1$ and so by Lemma \ref{l4}, $|f((2,1))|=1$.
It follows from Lemma \ref{l9} that $|f((1,2))|+|f((2,2))|=2$.
Consequently, we obtain $a_0+a_1+a_2=9$, a contradiction.
Thus $|f((1,1))|\ne1$.
Note that $|f((1,1))|\le2$ by Lemma \ref{l6} and $|f((1,1))|\ge1$ since $|f((0,1))|=|f((1,0))|=1$.
Therefore, we have $|f((1,1))|=2$ an so by Lemma \ref{l7}, $|f((1,2))|=|f((2,1))|=0$.
This implies that $|f((2,2))|\ge1$ and hence by Lemma \ref{l4}, we get $|f((2,2))|=1$.
Since $|f((1,2))|=|f((2,1))|=0$ and $|f((2,2))|=1$, we conclude from the definition of $\gamma_{tr3}(P_3\square P_n)$-function that  $|f((2,3))|\ge1$ and so by Lemma \ref{l4}, $|f((2,3))|=1$.
And it follows from Lemma \ref{l7} that $|f((1,3))|\ne2$.
Moreover, since $|f((0,3))|=1$ by Lemma \ref{l8} and $|f((1,2))|=0$, $|f((1,3))|\ge1$ and so we conclude from Lemma \ref{l6} that $|f((1,3))|=1$.
Repeating this process we can obtain that $|f((1,3t+1))|=2$ and $|f((1,3t+2))|=|f((2,3t+1))|=0$ for $0\le t\le n/3-1$ and $|f((i,j))|=1$ otherwise.
In particular, $|f((1,n-1))|=|f((2,n-2))|=0$ and $|f((2,n-1))|=1$. This is a contradiction to the fact that the subdigraph of $P_3\square P_n$ induced by the set $\{v\in V(P_3\square P_n):f(v)\ne\emptyset\}$ has no isolated vertex.
As a result, if $l=0$, then $\gamma_{tr3}(P_3\square P_n)\ge \left\lceil\frac{8n}{3}\right\rceil+1.$

To show the upper bound, we now provide a T$3$RDF $g:V(P_3\square P_n)\rightarrow\mathcal{P}(\{1,2,3\})$ as follows:
if $n\equiv0\pmod3$, then the function $g$ defined by
\begin{displaymath}
g((i,j))\\ =\left\{\begin{array}{ll}
\{1,2\},&\textrm{if}\ i=1\ \textrm{and}\ j=3t+1\ \textrm{for}\ 0\le t\le n/3-1,\\
\emptyset,&\textrm{if}\ i=1\ \textrm{and}\ j=3t+2\ \textrm{for}\ 0\le t\le n/3-2,\\
          & \textrm{or}\  i=2\ \textrm{and}\ j=3t+1\ \textrm{for}\ 0\le t\le n/3-1,\\
\{3\},&\textrm{otherwise},
\end{array}\right.
 \end{displaymath}
and so
$$\gamma_{tr3}(P_3\square P_n)\le\omega(g)=2\times\frac{n}{3}+\bigg(3n-\frac{n}{3}-\frac{2n}{3}+1\bigg)=\left\lceil\frac{8n}{3}\right\rceil+1,$$
and if $n\equiv1,2\pmod3$, then the function $g$ defined by
\begin{displaymath}
g((i,j))\\ =\left\{\begin{array}{ll}
\{1,2\},&\textrm{if}\ i=1\ \textrm{and}\ j=3t+1\ \textrm{for}\ 0\le t\le \lceil(n-1)/3\rceil-1,\\
\emptyset,&\textrm{if}\ i=1\ \textrm{and}\ j=3t+2\ \textrm{for}\ 0\le t\le \lceil(n-2)/3\rceil-1,\\
          & \textrm{or}\  i=2\ \textrm{and}\ j=3t+1\ \textrm{for}\ 0\le t\le \lceil(n-1)/3\rceil-1,\\
\{3\},&\textrm{otherwise},
\end{array}\right.
 \end{displaymath}
and so
\begin{eqnarray}
\gamma_{tr3}(P_3\square P_n)&\le&\omega(g)  \nonumber\\
&=&2\times\left\lceil\frac{n-1}{3}\right\rceil+\bigg(3n-\left\lceil\frac{n-1}{3}\right\rceil-\left\lceil\frac{n-2}{3}\right\rceil-\left\lceil\frac{n-1}{3}\right\rceil\bigg) \nonumber\\
&=&3n-\left\lceil\frac{n-2}{3}\right\rceil  \nonumber\\
&=&\left\lceil\frac{8n}{3}\right\rceil,  \nonumber
\end{eqnarray}
which completes our proof.
\end{proof}

\section*{\large Acknowledgements}

This study was supported by Research Foundation of Education Bureau
of Jiangxi Province of China (No. GJJ180374).

\end{document}